\documentclass{amsart}
\usepackage{colonequals,enumitem,amssymb,amsmath}
\usepackage{graphicx,caption,subcaption}
\usepackage{multirow}
\usepackage{hyperref}
\usepackage{systeme,mathtools}
\usepackage{xcolor}
\usepackage{tikz}
\makeatletter
\renewcommand*\env@matrix[1][*\c@MaxMatrixCols c]{%
  \hskip -\arraycolsep
  \let\@ifnextchar\new@ifnextchar
  \array{#1}}
\makeatother

\theoremstyle{plain}

\newtheorem*{cor}{Corollary}
\newtheorem*{defn}{Definition}
\newtheorem{claim}{Claim}
\newtheorem{problem}{Problem}

\title{Caps and Wickets}

\begin{document}

\begin{abstract}
Let $H_n^{(3)}$ be a 3-uniform linear hypergraph, i.e. any two edges have at most one vertex common. A special hypergraph, {\em wicket}, is formed by three rows and two columns of a $3 \times 3$ point matrix. In this note, we give a new lower bound on the Tur\'an number of wickets using estimates on cap sets. We also show that this problem is closely connected to important questions in additive combinatorics.
\end{abstract}

\author[]{Jakob F\"uhrer}
\address{Jakob F\"uhrer, Institute of Analysis and Number Theory, Graz University of Technology,
Kopernikusgasse 24/II,
8010 Graz, Austria.}
\email{jakob.fuehrer@tugraz.at}
\author{Jozsef Solymosi}
\address{Jozsef Solymosi, Department of Mathematics, University of British Columbia, Vancouver, Canada, and Obuda University, Budapest, Hungary}
\email{solymosi@math.ubc.ca}

\maketitle

\section{Introduction}

Gy\'arf\'as and S\'ark\"ozy asked for upper and lower bounds on the Tur\'an number of a linear hypergraph called the wicket. Although it was not obvious in the original formulation, it seems this problem is at the crossing point of important questions in additive combinatorics and extremal hypergraph theory.
To formulate the problem, let's begin with the basic definitions.

\begin{defn}
A hypergraph is linear if two edges intersect in at most one vertex.
\end{defn}

We will work with 3-uniform linear hypergraphs, i.e., where every edge has three vertices.

\begin{defn}
    The Tur\'an number of a linear 3-uniform hypergraph $F$, denoted by $ex_L(n,F)$, is the maximum number of edges of a 3-uniform linear hypergraph not containing a subgraph isomorphic to $F$.
\end{defn}

In \cite{GyS_1} Gy\'arf\'as and S\'ark\"ozy investigated $ex_L(n,F)$ for any $F$ with at most five edges. They had good estimates except for one configuration, called {\em wicket}. The wicket, denoted by $W$, is formed by three rows and two columns of a $3\times 3$ point matrix (Fig. \ref{Wicket}). Wickets are important structures in extremal hypergraph theory. Using a classical technique introduced by Ruzsa and Szemer\'edi in \cite{RSz}, we show the connection between its Tur\'an number and solution sets of equations like in (\ref{Ruzsa}). There are also connections to a conjecture of Gowers and Long, which we will describe later.

\begin{figure}[ht]
\centering
\includegraphics[scale=.3]{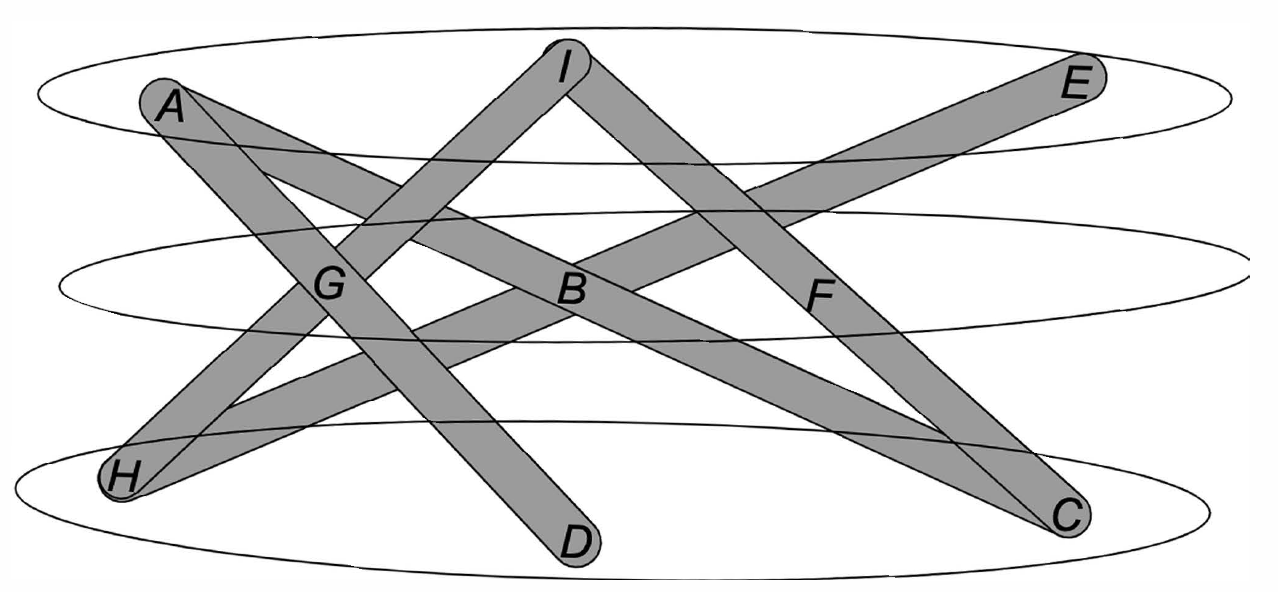}
\caption{The wicket is drawn as a three-partite hypergraph. If we add the edge spanned by vertices $D,E,F$, it is isomorphic to a $3\times 3$ grid.}
\label{Wicket}
 \end{figure}

\medskip
In a recent paper \cite{Soly}, it was proved that $ex_L(n,W)=o(n^2)$.
Only $ex_L(n,W)\geq cn^{3/2}$ type bounds were known for the lower bound \cite{GyS_1}. Our goal is to improve this bound.

\section{Caps}
In this section, we use cap sets to improve the lower bound on $ex_L(n,W)$. We also show that strong upper bounds on $ex_L(n,W)$ would lead to an improved bound on the size of cap sets in $\mathbb{F}_3^n$.

 \medskip
 
  The wicket has cycles of length four, so hypergraphs avoiding quadrilaterals are wicket-free, giving the $ex_L(n,W)\geq cn^{3/2}$ bound \cite{LV}. In our improvement, the new construction follows the steps of the classical work of Ruzsa and Szemer\'edi \cite{RSz}. In their construction, they defined a three-partite 3-uniform hypergraph. The three vertex classes, $A,B,C$, are three copies of $\mathbb{Z}/n\mathbb{Z}$. Let us suppose $S\subset\mathbb{Z}/n\mathbb{Z}$ is $AP_3$-free. In the hypergraph, three vertices $a\in A, b\in B$ and $c\in C$ are connected by an edge if there is an $s\in S$ such that $b=a+s$ and $c=a+2s$. In this setting, no six vertices carry three edges since it would rise to a solution of $s+t=2h$ (see in Fig. \ref{6,3}).

  \begin{figure}[ht]
\centering
\includegraphics[scale=.18]{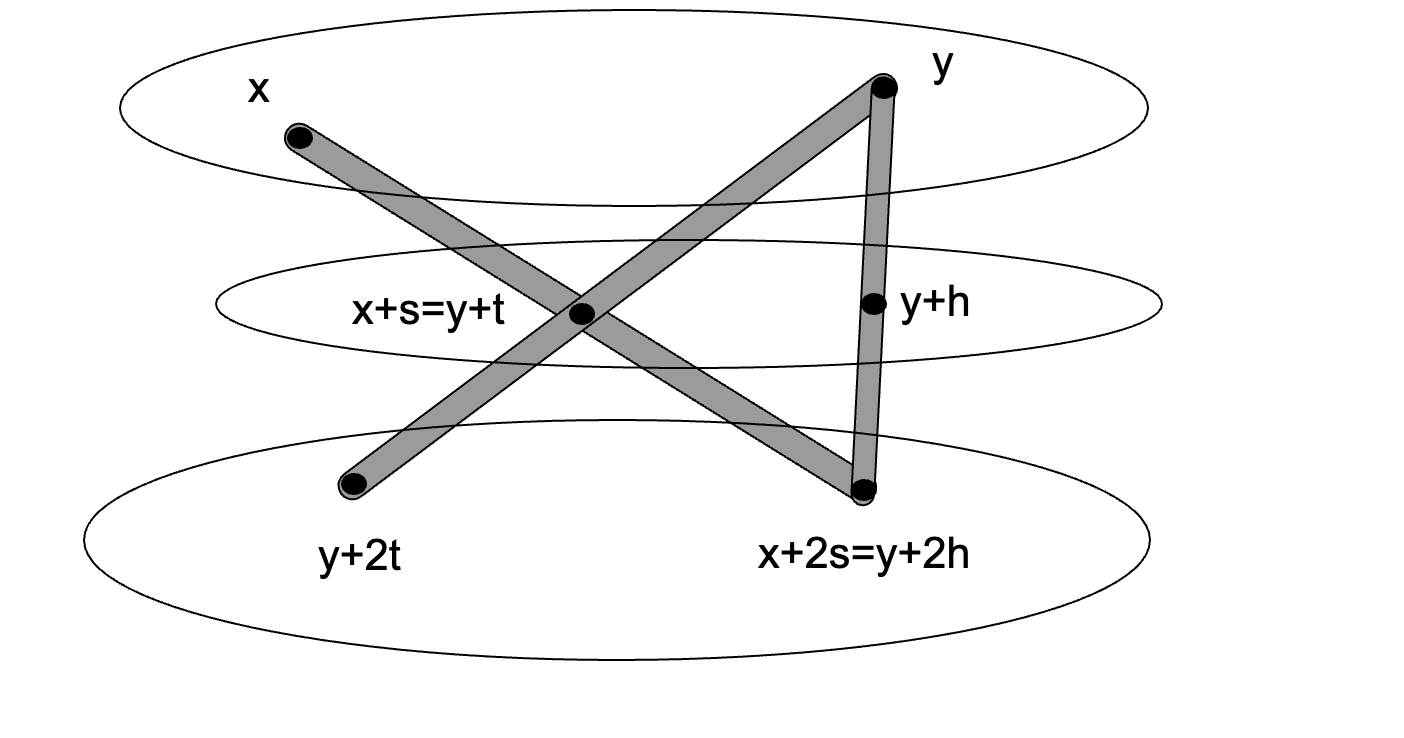}
\caption{After eliminating $x$ and $y$ we get $s+t=2h$}
\label{6,3}
 \end{figure}

Let's follow the same method as above, but here the three vertex classes, $A:=\mathbb{F}_3^n\times\{0\}$,  $B:=\mathbb{F}_3^n\times\{1\}$ and  $C:=\mathbb{F}_3^n\times\{2\}$, are three parallel hyperplanes in $\mathbb{F}_3^{n+1}$. 
Let $S\subset\mathbb{F}_3^n$ denote a maximal subset without 3-term arithmetic progressions (cap sets). It was proved recently in \cite{Nature} that there are cap sets of size $2.2202^n$, improving earlier bounds in \cite{Ed,Ty}. The $2.756^n$ upper bound on the size of the largest cap set was proved in \cite{EG}, following bounds on $AP_3$-free sets in $\mathbb{F}_4^n$ in \cite{CVP}.
In the hypergraph, denoted by $\mathcal{H}^{(3)}(A,B,C)$, three vertices $a\in A, b\in B$ and $c\in C$ are connected by an edge if there is an $s\in S':=S\times\{1\}$ such that $b=a+s$ and $c=a+2s$. As illustrated in Fig. \ref{Wicket_2}, a wicket defines four linear equations.

\begin{figure}[ht]
\centering
\includegraphics[scale=.15]{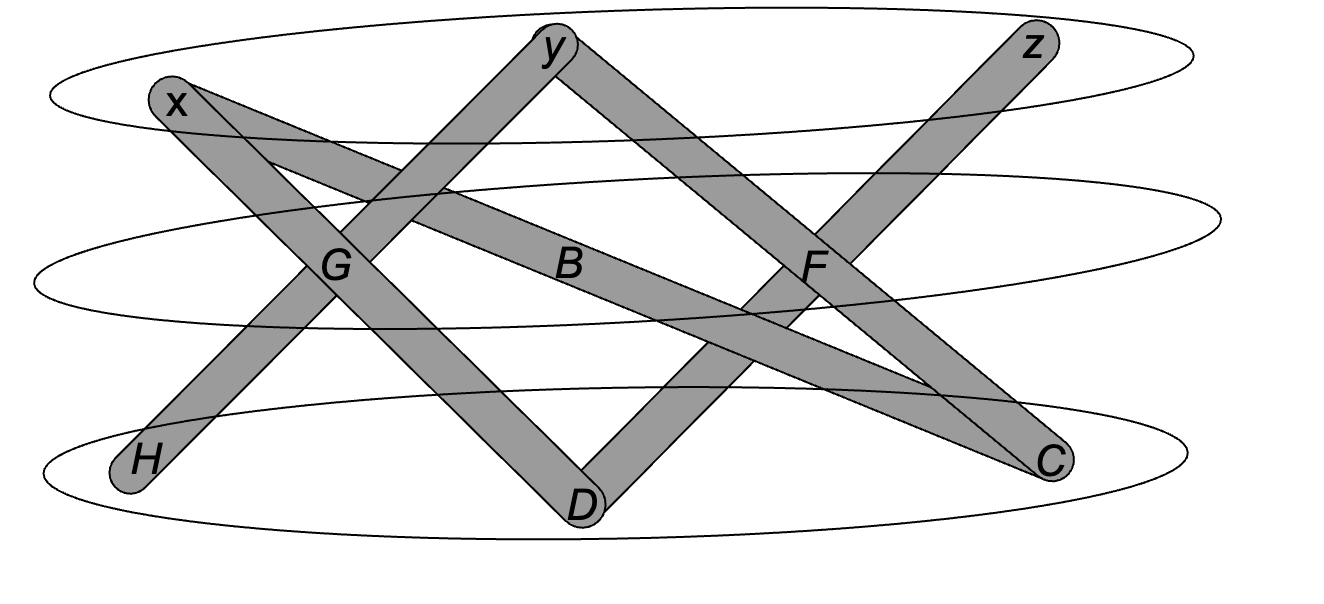}
\caption{In the wicket, $G=x+s=y+t, D=x+2s=z+2v, F=y+u=z+v$ and $C=x+2w=y+2u$. }
\label{Wicket_2}
 \end{figure}

\begin{align*}
  &x+s=y+t \\ 
  &x+2s=z+2v \\ 
  &y+u=z+v \\
  &x+2w=y+2u
  \end{align*}

After eliminating $x,y$ and $z$, in $\mathbb{F}_3^{n+1}$ we get the independent equations $w+v=2t$  and $s+t=u+v$.
The first equation has no solution for distinct $t,v,w$ values in $S'$ since it is a cap set in $\mathbb{F}_3^{n+1}$. Therefore, the only wickets in  $\mathcal{H}^{(3)}(A,B,C)$ are the ones coming from $t=v=w$ and $s=u$. Note here that $t\neq s$ as otherwise every edge would be equal; consequently, every wicket corresponds to $5$ lines in a $2$-dimensional affine subspace. Moreover, each of these affine subspaces contains $6$ lines corresponding to $t$ and $s$, each $5$ of them defining a wicket (see Fig. \ref{fig.plane}).

Each wicket $W'$ in $\mathcal{H}^{(3)}(A,B,C)$ intersects at most $30|S|$ other wickets: each edge $e$ of $W'$, together with an element of $s'\in S$, spans a $2$-dimensional affine subspace in which at most $6$ wickets, derived from $s'$ and the direction of $e$, intersect $W'$.

\begin{figure}[ht]
\centering
\includegraphics[scale=.5]{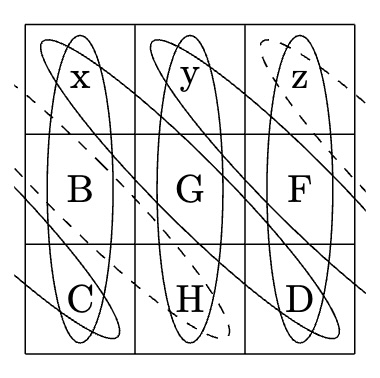}
\caption{Every wicket in $\mathcal{H}^{(3)}(A,B,C)$ lies in a $2$-dimensional affine subspace.}
\label{fig.plane}
\end{figure}

We now colour each edge of $\mathcal{H}^{(3)}(A,B,C)$ independently at random in $k:=\left(120|S|\right)^{1/4}$ colours, each with probability $1/k$. The probability that a wicket in $\mathcal{H}^{(3)}(A,B,C)$ is monochromatic is  $(1/k)^4$, independently of all but at most $30|S|$ other wickets. Using The Lov\'asz local lemma \cite{LLL}, we get a colouring of $\mathcal{H}^{(3)}(A,B,C)$  without monochromatic wickets with positive probability. Choosing the largest colour class, we get a wicket-free 3-uniform linear hypergraph with $3^{n+1}$ vertices and at least $3^n|S|/k\geq (3\cdot 2.2202^{3/4})^n/120^{1/4}$ edges, which gives the improved bound $ex_L(m,W)\geq m^{1.544}$.
The construction also avoids the $(6,3)$-configuration from Fig. \ref{6,3}.

\medskip
Gowers and Long conjectured (Conjecture 4.6 in \cite{GL}) that there is a $c>0$ such that if a 3-uniform linear hypergraph on $n$ vertices has no nine vertices spanning at least five edges, then its number of edges is $O(n^{2-c})$. 

\begin{claim}\label{GL}
Every 3-partite, 3-uniform linear hypergraph on nine vertices with at least five edges contains a wicket or a $(6,3)$-configuration.
\end{claim}

\begin{proof}
The claim follows from Theorem 2.2 in \cite{GyS_1}, but we include a simple argument for completeness.
If a vertex is contained in three edges forming a seven-vertex star, then any of the remaining two edges shares at most one of these seven since connecting two would form a triangle (a $(6,3)$-configuration). Then, the remaining two edges are incident to at most two vertices of the seven and intersect in at most one because it is a linear hypergraph. Then, the minimum number of vertices is $7+1+2=10$, which is not allowed. Therefore, every vertex has degree one or two. It also means that all three partitions have three vertices. Six vertices have degree two. The remaining three have degree one, one in each partition class. Select the two degree-two vertices in the middle vertex class, with edges $e,e'$ and $f,f'$ (see Fig \ref{twox_fig}). 

\begin{figure}[ht]
\centering
\includegraphics[scale=.4]{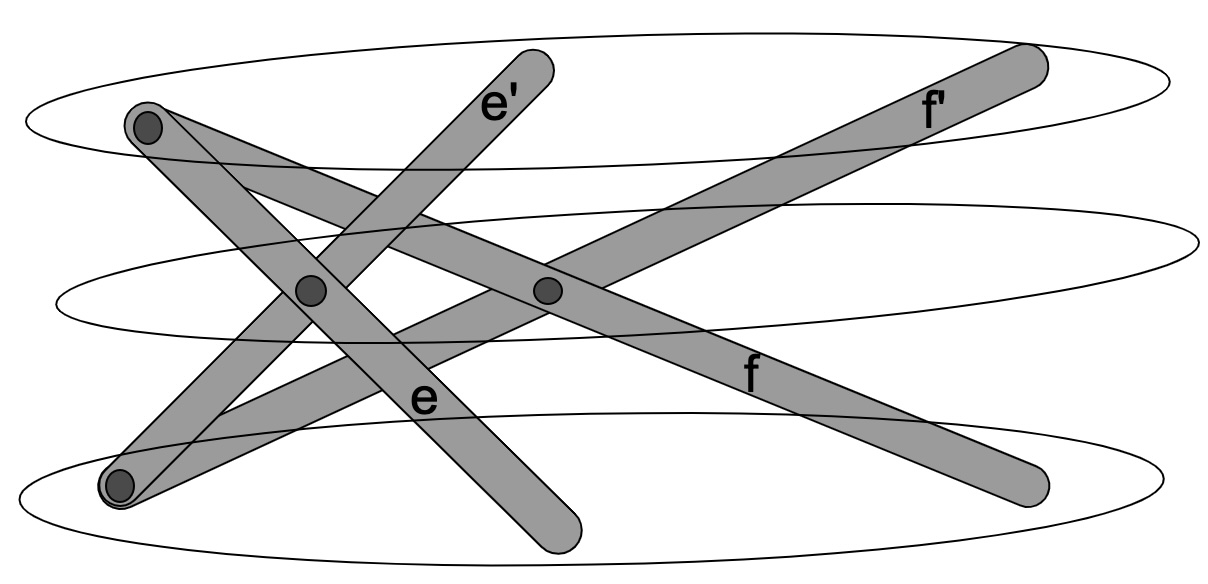}
\caption{$e, f$ and $e',f'$ share a vertex}
\label{twox_fig}
 \end{figure}
 
We can assume that $e, f$ and $e',f'$ share a vertex. Then the remaining edge has a degree-one vertex in the middle vertex class and connects two degree-one vertices of the subgraph spanned by $e,e',f,f'$, forming a wicket. 
\end{proof}

Claim \ref{GL} implies that in the Gowers-Long conjecture $c\leq 0.456$, improving the previous $c\leq 0.5$ bound for this case.

Based on some computational evidence, Tyrrell conjectured that there are cap sets of size $2.233^n$ (Conjecture 4.3 in \cite{Ty}). Such cap sets would slightly improve the bound beyond $ex_L(m,W)\geq m^{1.548}$.

We also record the following corollary of our construction:
\begin{cor}
Any upper bound of the form $ex_L(m,W)\leq m^{2-c}$ would lead to an upper bound of $3^{\frac{4}{3}(1-c)n}$ for the size of a cap set in $\mathbb{F}_3^n$.
\end{cor}

For example, an $ex_L(m,W)\leq m^{1.69}$ bound would improve the cap set bound of Ellenberg and Gijswijt \cite{EG}.

\section{Related Problems}

In this section, we list three problems. Each can potentially improve the previous bound, maybe even to $ex_L(m,W)\geq m^{2-\varepsilon}$.

\medskip
The hypergraph construction was based on the observation that the 
\begin{equation}\label{Ruzsa}
    3x+y=2z+2w 
\end{equation}
equation has no non-trivial solutions in cap sets in $\mathbb{F}_3^n$. This equation is a classic example from Ruzsa's work on solution sets of linear equations in integers \cite{Ru}.

\begin{problem}[Ruzsa (1993)]
What is the size of the largest subset of the first $n$ natural numbers without non-trivial solutions to (\ref{Ruzsa})?   
\end{problem}

It is possible that there are sets $S\subset [n]$ without non-trivial solutions, with $|S|=n^{1-o(1)}$. It would give the $ex_L(m,W)= m^{2-o(1)}$ bound as conjectured in \cite{Soly}. But for further improvements, finding any abelian group with a large subset without a non-trivial solution to (\ref{Ruzsa}) or a similar linear equation would be enough.

\medskip
A possible variant of the Ruzsa-Szemer\'edi construction is the following: Instead of defining the edges as $x,x+s,x+2s$, we can use a different parameter, $\alpha$ and define the edges as $x,x+s,x+\alpha s$. Instead of $\mathbb{F}_3^n$, let's define $\mathcal{H}^{(3)}(A,B,C)$ on three vertex classes as three copies of 
$\mathbb{Z}/n\mathbb{Z}$ where $n=k^2-k+1$ with some $k\geq 2, k\in \mathbb{N}$.
The edges are defined by a set $S\subset \mathbb{Z}/n\mathbb{Z}$, which has no non-trivial solution to the $kx-(k-1)y=z$ equation.
Three vertices $a\in A, b\in B$ and $c\in C$ are connected by an edge if there is an $s\in S$ such that $b=a+s$ and $c=a+ks$.
This definition gives a linear hypergraph since both $k$ and $k-1$ are relative prime to $n$. Therefore, any two vertices of an edge determine the third one uniquely. The four equations determined by the vertices of the wicket are

\begin{align*}
  &x+s=y+t \\ 
  &x+ks=z+kv \\ 
  &y+u=z+v \\
  &x+kw=y+ku
  \end{align*}

After eliminating $x,y,z$ and $s$, we get the 

\[(k^2-k+1) u-(k^2-k) w +(k-1) v=k t\]

equation, which gives

   \begin{equation*}
    kt-(k-1)v\equiv w \pmod{k^2-k+1}.
   \end{equation*}

Since $S$ has no solution to this equation, 
$\mathcal{H}^{(3)}(A,B,C)$ is wicket-free. 

\begin{problem}
   What is the largest set $S\subset \mathbb{Z}/n\mathbb{Z}$, which has no non-trivial solution to the 
   \begin{equation}\label{modeq}
    kx-(k-1)y\equiv z \pmod{n}
   \end{equation}
   equation, where 
   $n=k^2-k+1$ for some large $k\in \mathbb{N}$?  
   \end{problem}

We considered mod $n=k^2-k+1$ to ensure that one more variable ($u$ in this case) will vanish after the elimination of variables, so we left with three unknowns in equation (\ref{modeq}). There are different ways to achieve this. It gives a nice related geometric problem if we formulate a construction using Eisenstein integers \cite{We}. 

\medskip
Let $\omega=\frac{-1+i\sqrt{3}}{2}$. The Eisenstein integers are the complex numbers of the form $a+\omega b$ where $a,b\in \mathbb{Z}$. In our new $\mathcal{H}^{(3)}(A,B,C)$ hypergraph the vertex sets are Eisenstein integers, $$E_n=A=B=C=\{a+\omega b : N(a+\omega b)=a^2+b^2\leq n\}.$$

The hypergraph has $\Delta n^2$ vertices (for sharp estimates on  $\Delta$ we refer to \cite{LP}). We use a subset of $E_n$ to define the edges. The Eisenstein integers are points of a triangular grid in the complex plane. Let $S_n\subset E_n$ denote the largest subset not containing three points of an equilateral triangle.
Three vertices $a\in A, b\in B$ and $c\in C$ are connected by an edge if there is an $s\in S_n$ such that $b=a-s$ and $c=a+\omega s$.
This definition gives a linear hypergraph since the Eisenstein integers form a commutative ring of algebraic integers in the algebraic number field $\mathbb{Q}(\omega)$. Therefore, any two vertices of an edge determine the third one uniquely in the hypergraph. The four equations determined by the vertices of the wicket are
\begin{align*}
  &x-s=y-t \\ 
  &x+\omega s=z+\omega v \\ 
  &y-u=z-v \\
  &x+\omega w=y+\omega u
  \end{align*}

After eliminating $x,y,z$ and $s$ we get the $2 t=-i \sqrt{3} v+v+i \sqrt{3} w+w$ equation, or equivalently,
\begin{equation}\label{triangle}
    t-w=\omega(w-v).
\end{equation}

But a non-trivial solution to equation (\ref{triangle}) would raise an equilateral triangle by points $t,v,w\in S_n$, which was not allowed.

\begin{problem}
    What is the largest subset of a triangular grid without three points forming an equilateral triangle?
\end{problem}

It is easy to use a Behrend-type construction \cite{Be} to find a large subset avoiding an equilateral triangle with a fixed orientation; however, it seems like a difficult problem to estimate the size of the largest subset avoiding equilateral triangles in any orientation.

\begin{figure}[ht]
\centering
\includegraphics[scale=.2]{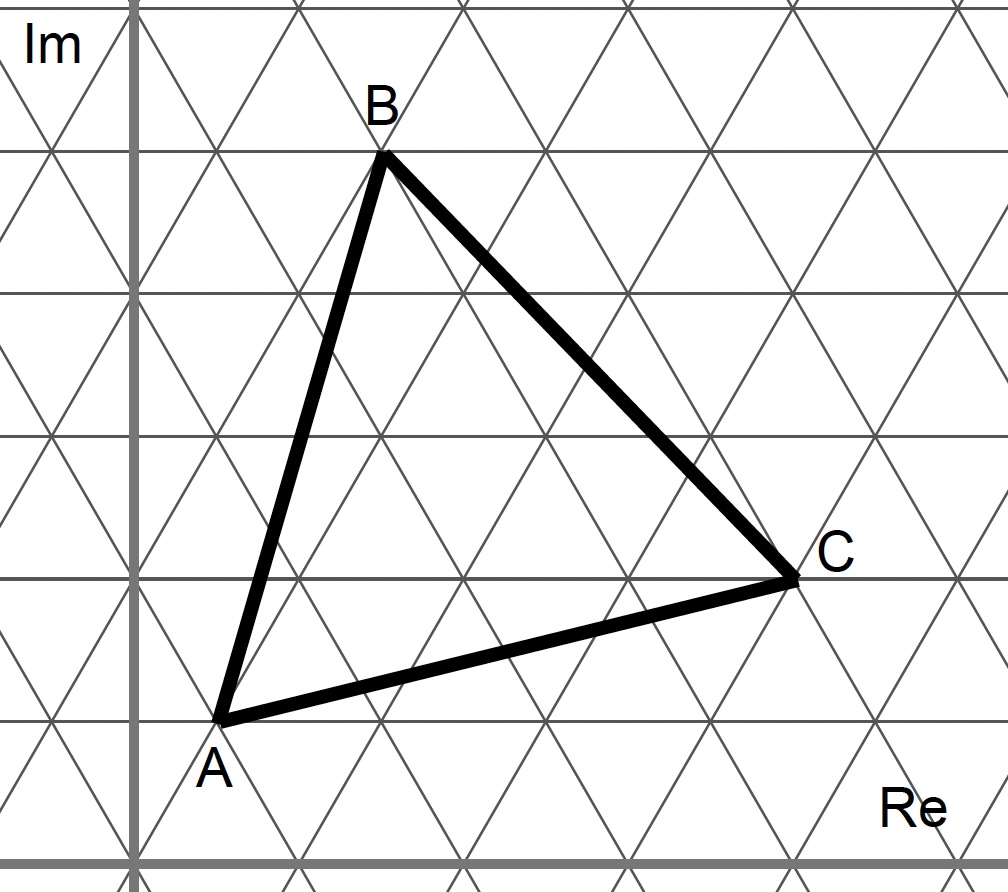}
\caption{An equilateral triangle formed by Eisenstein integers }
\label{Triangle}
 \end{figure}

\section{Acknowledgements}
J. F. was supported by the Austrian Science Fund (FWF) under
the project W1230. J. S. was partly supported by an NSERC Discovery grant, and OTKA K  grants no. 133819. The authors are thankful to Ben Green for the useful suggestions resulting in Problem 3.


\begin{thebibliography}{100}

\bibitem{Be}
F. Behrend. 
On sets of integers which contain no three terms in arithmetic progression, 
Proc. Nat. Acad. Sci., 32:331--332, 1946

\bibitem{CVP}
 Croot, Ernie; Lev, Vsevolod; Pach, Peter (2017), Progression-free sets in 
${\displaystyle Z_{4}^{n}}$ are exponentially small, Annals of Mathematics, 185 (1): 331--337

\bibitem{Nature} Romera-Paredes, B., Barekatain, M., Novikov, A. et al. 
Mathematical discoveries from program search with large language models. 
Nature 625, 468--475 (2024).

\bibitem{Ed} 
Edel, Y. Extensions of generalized product caps. Des. Codes Cryptogr. 31, 5–14 (2004).

\bibitem{EG}
Ellenberg, Jordan S.; Gijswijt, Dion (2017), On large subsets of 
${\displaystyle \mathbb {F} _{q}^{n}}$ with no three-term arithmetic progression, 
Annals of Mathematics, 185 (1): 339--343


\bibitem{GL}
William Gowers and Jason Long,  The length of an $s$-increasing sequence of $r$-tuples. 
Combinatorics, Probability and Computing, (2021) 30(5), 686--721.


\bibitem{GyS_1} Andr\'as Gy\'arf\'as and G\'abor N. S\'ark\"ozy,
The linear Turán number of small triple systems or why is the wicket interesting?
Discrete Mathematics,
Volume 345, Issue 11,
2022,

\bibitem{LP} Peter D Lax, Ralph S Phillips,
The asymptotic distribution of lattice points in Euclidean and non-Euclidean spaces,
Journal of Functional Analysis,
Volume 46, Issue 3,
1982,
Pages 280--350,


\bibitem{LV} Felix Lazebnik and Jacques Verstra\"ete,
On Hypergraphs of Girth Five,
The Electronic Journal of Combinatorics,
Volume 10 (2003) Article Number: R25

\bibitem{LLL} Erd\H{o}s, Paul and Lov\'asz, L\'aszló (1975). Problems and results on 3-chromatic hypergraphs and some related questions. 
In A. Hajnal; R. Rado; V. T. S\'os (eds.). Infinite and Finite Sets (to Paul Erdős on his 60th birthday). Vol. II. North-Holland. pp. 609--627.

\bibitem{Ru}
Imre Z. Ruzsa,
Solving a linear equation in a set of integers I. Acta Arithmetica 65.3 (1993): 259--282.


\bibitem{RSz}
Imre~Z. Ruzsa and Endre Szemer\'edi.
\newblock Triple systems with no six points carrying three triangles.
\newblock In {\em Colloq. Math. Soc. J. Bolyai  Combinatorics II: 939--945, 1978}

\bibitem{Soly}
Jozsef Solymosi, Wickets in 3-uniform hypergraphs,
Discrete Mathematics,
Volume 347, Issue 6,
2024,

\bibitem{Ty} Tyrrell, F. New lower bounds for cap sets. Discrete Analysis (2023)

\bibitem{We} Weisstein, Eric W. "Eisenstein Integer." From MathWorld--A Wolfram Web Resource. https://mathworld.wolfram.com/EisensteinInteger.html


\end{thebibliography}
\end{document}